\documentclass{amsart}
\usepackage{amsmath,amssymb}
\usepackage[all]{xy}

\theoremstyle{plain}
\newtheorem{introtheorem}{Theorem}
\newtheorem{theorem}{Theorem}[section]
\newtheorem{proposition}[theorem]{Proposition}
\newtheorem{lemma}[theorem]{Lemma}

\newtheorem*{proposition*}{Proposition}

\theoremstyle{definition}
\newtheorem{definition}[theorem]{Definition}
\newtheorem{example}[theorem]{Example}

\theoremstyle{remark}
\newtheorem{remark}[theorem]{Remark}

\newtheorem*{thanks*}{Acknowledgements}

\newcommand{\secref}[1]{Section~\ref{#1}}
\newcommand{\thmref}[1]{Theorem~\ref{#1}}
\newcommand{\propref}[1]{Proposition~\ref{#1}}
\newcommand{\lemref}[1]{Lemma~\ref{#1}}

\newcommand{\exref}[1]{Example~\ref{#1}}

\def\B{{\mathcal B}}

\def\Ad{{\mathrm{Ad}}}
\def\Pad{{\mathrm{Pad}}}

\def\Z{{\mathbb Z}}
\def\Q{{\mathbb Q}}

\def\G{\mathcal G}

\def\rat{{(\Q)}}
\def\map{\mathrm{map}}

\def\cat0{\mathrm{cat}_0}

\begin{document}

\title[Fibrewise Rational H-Spaces]
{Fibrewise Rational H-Spaces}

\author{Gregory  Lupton}

\address{Department of Mathematics,
           Cleveland State University,
           Cleveland OH 44115}

\email{G.Lupton@csuohio.edu}

\author{Samuel Bruce Smith}

\address{Department of Mathematics,
   Saint Joseph's University,
   Philadelphia, PA 19131}

\email{smith@sju.edu}

\date{\today}



\begin{abstract}
We  prove   fibrewise versions of  classical theorems of Hopf and Leray-Samelson.  Our results imply  the fibrewise H-triviality after rationalization of a certain class of fibrewise H-spaces. They apply, in particular, to universal adjoint bundles.  From this, we may retrieve a result of Crabb and Sutherland \cite{CS}, which is used there as a crucial step in establishing their main finiteness result.     
\end{abstract}

\maketitle

\section{Introduction}%
\label{sec:intro}
Let  $G$ be a connected H-space of the homotopy type of a finite CW complex.    
Hopf's theorem gives  an  isomorphism of graded algebras $$H^*(G; \Q) \cong \land W$$ where $W$ is a finite-dimensional, oddly graded,    rational vector  space \cite[Cor.3.8.13]{Wh}. Topologically, Hopf's theorem implies  a homotopy equivalence
$$G_\Q \simeq  \prod_i K(W_{2i+1}, 2i+1).$$   
Here, $G_\Q$ denotes the rationalization of $G$, in the sense of \cite{HMR}. 

The Leray-Samelson theorem  strengthens these statements when the multiplication for $G$  is homotopy associative. In this case, the isomorphism $H^*(G; \Q) \cong \land W$ is one  of Hopf algebras,  where $\land W$ has the standard (co)multiplication $C\colon \land W \to \land W \otimes \land W$ given by $C(w) = w\otimes 1 + 1\otimes w$ for $w \in W$  \cite[Th.3.8.14]{Wh}.  Further, the homotopy equivalence $ G_\Q \simeq  \prod_i K(W_{2i+1}, 2i+1)$  above is realized by an H-equivalence, where the product of Eilenberg-Mac\,Lane spaces comes equipped with the unique homotopy-associative  and homotopy-commutative multiplication.

In this paper, we prove fibrewise versions of the Hopf and Leray-Samelson  theorems  using Sullivan  model techniques. We work in the setting of fibrewise pointed spaces and maps  over a fixed base $B$ as in the text \cite{CJ}.    We assume all ordinary spaces are of the homotopy type of (well-pointed) CW complexes of finite type.  We say such a space is {\em homotopy finite} if it is homotopy equivalent to a finite complex.  A  {\em fibrewise space} $X$ means a fibration $p\colon X \to B$, and a {\em fibrewise basepoint} a (choice of) section $\sigma \colon B \to X$ of $p$.  Then a fibrewise pointed space consists of $X$ together with fixed choices of $p$ and $\sigma$.  (Fibrewise pointed)  maps between fibrewise pointed spaces  are maps over $B$ (via $p$) and under $B$ (via $\sigma$).  Homotopies are fibrewise pointed and denoted $\sim_B^B.$ The relation between spaces over and under $B$   of being fibrewise pointed homotopy equivalent will be denoted by $\simeq_B^B$.
A \emph{fibrewise multiplication} on $X$  is a fibrewise pointed map  $m \colon X \times_B X \to X$  where $X \times_B X \to B$ is the fibre product.
A  \emph{fibrewise H-space}  $X$ is a fibrewise pointed space  $X$ together with a fibrewise multiplication $m$ such that the maps $m \circ (1 \times c) \circ \Delta$ and $m \circ (c \times 1) \circ \Delta$ are both fibrewise pointed homotopic to the identity of $X$, where $\Delta \colon X \to X\times_B X$  is the diagonal and $c$ is the fibrewise null map $c = \sigma\circ p\colon X \to X$. Fibrewise homotopy associativity is defined, in the usual way, with the appropriate identity holding up to fibrewise pointed homotopy.   See \cite{CJ} for a basic discussion of fibrewise H-spaces.

Suppose $X, m$ is a fibrewise H-space and let $G$ denote the fibre over a basepoint (in the ordinary sense) of B.  Then $G$ is  an H-space with the restricted multiplication written $m_G$. We say  $X$ is {\em fibrewise trivial} (as a fibrewise pointed space) if there is a fibrewise pointed equivalence $f \colon X \to B \times G.$   Here it is understood that $B \times G$ is the fibrewise space with projection $p_1\colon B \times G \to B$ and basepoint $i_1 \colon B \to B \times G$.    
A fibrewise H-space $X$  admits  a  {\em fibrewise rationalization}  $X_\rat$ over $B$ and a localization map $X \to X_\rat$ which is a map over $B$ \cite{May}.  We say $X$ is \emph{rationally fibrewise trivial}  if $X_\rat$ is fibrewise trivial, that is, if $X_\rat$ is fibrewise pointed homotopically equivalent to a trivial fibration over $B$.   Our first result is the following which may be viewed as a fibrewise version of Hopf's theorem: 

\begin{introtheorem}\label{main}
Let $X$ be a fibrewise    H-space over  $B$  with fibre $G$, with $X$, $B$, and $G$ all of finite type, and either $X$  or $G$  homotopy finite.   Suppose: \begin{itemize} \item[(1)] $B$ is  simply connected with zero odd-dimensional rational cohomology;  \item[(2)] $G$ has zero even-dimensional rational homotopy; and  \item[(3)]  $X$ is nilpotent as a space.  \end{itemize}  Then $X$ is rationally fibrewise  trivial. 
\end{introtheorem}
If $X$ is a fibrewise H-space that is rationally fibrewise trivial, then we have 
$$H^*(X; \Q) \cong H^*(B;   \Q) \otimes \land W \hbox{\, and \, } X_\rat \simeq^B_B B \times K,$$
where the cohomology isomorphism is one of algebras, and $K$ is a product of rational Eilenberg-Mac\;Lane spaces. 

Our main result strengthens Theorem \ref{main} if we add the hypothesis of associativity on the fibrewise multiplication.  It may be viewed as a fibrewise version of the Leray-Samelson theorem.
We  say   $X, m$  is  {\em fibrewise H-trivial} (as a fibrewise H-space) if   $X$ is  fibrewise trivial via a fibrewise pointed, fibrewise equivalence $f \colon X \to B \times G$    satisfying  $$f \circ m \sim^B_B (1 \times_B m_{G}) \circ (f \times_B f) \colon X \times_B X \to B \times G.$$
Here   $1 \times_B  m_G$ is the fibrewise  multiplication on the product $B \times G$ induced by $m_G$.   
The multiplication $m \colon X\times_B X \to X$  induces  a multiplication $ m_{(\Q)} \colon  X_\rat \times_B X _\rat \to X_\rat$ and the fibrewise localization $\ell_{(X)}\colon X \to X_{(\Q)}$ gives a fibrewise H-map $X, m \to X_\rat, m_\rat$.   We say $X, m$ is  \emph{rationally fibrewise H-trivial} if $X_{(\Q)}, m_\rat$ is    fibrewise H-trivial.

\begin{introtheorem}  \label{main1} Let $X$ be a fibrewise homotopy associative   H-space over  $B$   satisfying the hypotheses in Theorem \ref{main}.   Then $X$ is rationally fibrewise H-trivial. 
\end{introtheorem}

 If   $X, m$ is rationally fibrewise $H$-trivial then the isomorphism $H^*(X; \Q) \cong H^*(B;   \Q) \otimes \land W$ above  is one of Hopf algebras over $H^*(B; \Q)$.  The equivalence $X_\rat  \simeq^B_B B \times K$   is via a fibrewise H-map.

 Theorems  \ref{main} and \ref{main1}  are      deduced from   algebraic results on fibrewise multiplications of models as we briefly explain now.     The fibration $p
\colon X \to B$ has a {\em relative   model} $\B \to \B \otimes \land
W$ where  $\B$ is a differential graded (DG) algebra model for $B$ and $\land W$ is the free  graded commutative algebra on $W \cong \pi_*(G) \otimes \Q$ with trivial differential (a Sullivan model for the H-space $G$) \cite[Pro.15.6]{FHT}. The fibrewise multiplication is modeled by a comultiplication map
$C \colon \B \otimes \land W \to \B \otimes \land W \otimes \land
W$
that is the identity on $\B$ and takes the form
$$C(w) = w + w'  + C_2(w) + C_3(w) + \cdots$$
for $w \in W$.  Here $C_r(w)$ denotes a term in $\B \otimes
\sum_{s=1}^{s=r-1}( \land^s W \otimes \land^{r-s} W)$.  Given $w \in W,$ we write   $w$ for  $w \otimes 1$ and $w'$ for $1 \otimes w.$   
In Theorem \ref{Hopf}, we prove that when $\B$ has evenly graded cohomology and $W$ is oddly-graded then the relative model is equivalent to one with a trivial differential on $W$.  Taking $\B$ trivial retrieves the ordinary (non-fibrewise) Hopf theorem. 
In Theorem \ref{LS}, we prove that, under these same hypotheses,   $C$ associative implies that $C$ is equivalent to the {\em standard comultiplication}:   $C_0(w) = w + w'$ for $w \in W.$ Here taking $\B$  to be  trivial retrieves
the ordinary Leray-Samelson theorem.

 Theorem \ref{main1}     retrieves   the following    result of   Crabb and Sutherland.  Let $G$ be a topological group and  $p_G \colon EG \to BG$ the universal principal $G$-bundle.   Form the {\em adjoint bundle}    
$$\Ad(p_G) \colon EG \times_G G  \to BG$$ where $G$ acts on itself by the adjoint action and $EG \times_G G $  is the quotient of $EG \times G $ by the diagonal action. This is a fibrewise group with the multiplication induced by that on the fibres.
  
\begin{proposition*}[Crabb-Sutherland {\cite[Pro.2.2]{CS}}] \label{CS}      Let $G$ be a compact, connected Lie group.  Then  $\Ad(p_G)$ is rationally fibrewise H-trivial.  
\end{proposition*}

\noindent{}This result forms the basis for the finiteness results of \cite{CS}.  We discuss this type of application further in \secref{sec:applications} below.
 
The paper is organized as follows.  In Section \ref{sec:prelim},  we consider preliminary issues  on modeling fibrewise H-spaces with Sullivan DG algebras.  In Section \ref{sec:LS}, we prove  fibrewise  versions of Hopf's theorem (Theorem \ref{Hopf}) and the Leray-Samelson theorem (Theorem \ref{LS}).  The former is an extension of an elementary argument for the original Hopf theorem using Sullivan models. The proof of Theorem \ref{LS} is considerably more involved and follows the line of argument of \cite{AL}. Theorems \ref{main} and \ref{main1} are direct consequences. 
    In Section \ref{sec:applications},  we apply Theorem \ref{main1} to   obtain  the rational fibrewise H-triviality of   a certain class of fibrewise groups admitting a universal example.     We deduce 
   the rational H-commutativity of the corresponding groups of sections, a class of groups including and generalizing the gauge groups of principal bundles. 

\begin{thanks*}
We thank the referee for correcting an oversight in the proof of \thmref{D=0} and also for suggesting a way to relax the hypotheses of \thmref{Hopf} and \thmref{LS} into their present form.
\end{thanks*}

\section{Rationalization of   Fibrewise H-spaces} \label{sec:prelim}

Our purpose in this section is to give    criteria on Sullivan models for fibrewise rational triviality and fibrewise rational H-triviality to hold.       First, we consider localization.   Note that, in the following, we are using localization in the sense of fibrewise localization as well as in the sense of localization of a nilpotent space.  Suppose $X$ is  just a fibrewise  pointed space $p \colon X \to B$ with basepoint $\sigma \colon B \to X$.  Assume is $X$ nilpotent as a space and $B$ is simply connected. Write $F$ for the fibre which is then  a nilpotent space by \cite[Th.2.2.2]{HMR}.  These hypotheses   imply the existence of a fibre square:   
$$ \xymatrix{ X \ar[r]^{\ell_X}  \ar[d]_p & X_\Q \ar[d]_{p_\Q} \\
 B \ar@/_/[u]_{\sigma} \ar[r]^{\ell_B} & B_\Q \ar@/_/[u]_{\sigma_\Q} }$$
in which the   horizontal maps are rationalizations (see \cite[p.67]{HMR}).   
The map $\sigma_{\Q}$ is the rationalization of $\sigma$ and is a section up to homotopy. 

These hypotheses on $X$  also ensure that $X$ admits a fibrewise rationalization
$$ \xymatrix{ X \ar[rr]^{\ell_{(X)}} \ar[dr]_p && X_{(\Q)} \ar[dl]^{p_{(\Q)}} \\
&  B \ar@/_/[ul]_{\sigma} \ar@/^/[ur]^{\sigma_{(\Q)}} &  }$$
by  \cite{May} (see \cite{L} also).    Here $\sigma_{(\Q)} =  \ell_{(X)}\circ \sigma$.  The map induced on fibres by $\ell_{(X)}$  is a rationalization map $F \to F_\Q$.  The latter construction has a universal property:  Let $g \colon X \to Y_\rat$ be  a fibrewise map over $B$ where $Y_\rat$ has fibre a rational space.  Then there is a map $h \colon X_\rat \to Y_\rat$ unique up homotopy over $B$, such that $g \sim_B h \circ \ell_{(X)}$ and $h$  \cite[Pro.6.1]{L}.  As   a consequence of this uniqueness,   we obtain $X_\rat$ as the pull-back of $p_\Q$ by $\ell_B$. 
It follows  that
$$\hbox{$  X_\Q$ fibrewise trivial} \ \ \Longrightarrow  \ \ \hbox{$X_\rat$ fibrewise trivial}.$$ 
Next  suppose $X$ is a  fibrewise H-space.   Then the fibrewise multiplication $m$ on $X$ induces one $m_\rat$ on $X_\rat$ that is unique up to fibrewise equivalence.    Similarly,  by universality  $m$  induces a  fibrewise multiplication $m_\Q$ on $X_\Q$  (although $m_\Q$ is only unique up to homotopy equivalence).  Since  $p_\rat \colon X_\rat \to B$ is the pull-back of $p_\Q \colon X_\Q \to B_\Q$ by $\ell_\Q \colon B \to B_\Q$   we see that  $m_\rat$ is equivalent to  the pull-back multiplication induced by $m_\Q.$
Consequently, we have:
$$\hbox{$ X_\Q , m_\Q$ fibrewise H-trivial} \ \ \Longrightarrow  \ \ \hbox{$ X_\rat, m_\rat$ fibrewise H-trivial}.$$
The two implications highlighted in the preceding discussion may be established as follows.  In the first case, from the (homotopy) pullback displayed above, standard arguments obtain a fibrewise pointed map $B\times F_\Q \to X_{(Q)}$ that is an (ordinary) homotopy equivalence.  Then this map is actually a fibrewise pointed homotopy equivalence, from, e.g. \cite[I.Th.13.2, II.Th.1.29]{CJ} (see also \cite{Egg}).  A similar line of argument may be applied to obtain the second implication.

Turning to Sullivan models,  
 let $(\land V, d_B)$   and $(\land W, d_F)$ denote, respectively,  the minimal Sullivan models for $B$ and $F.$  The fibration    $p \colon X \to B$ has a   {\em  relative minimal model}   (see \cite[Sec.15]{FHT}).  This   is an  inclusion of DG algebras $$I \colon (\land V, d_B) \to (\land V \otimes \land W, D).$$
 The DG algebra $ (\land V \otimes \land W, D)$ is a Sullivan  model for $X$ in the sense of \cite[Sec.12]{FHT}.  
 The differential $D$   satisfies   $D(w) - d_F(w) \in \land V^+ \otimes \land W $ for $w \in W$.  Furthermore, $W$ admits an ordered  basis $\{w_i\}$ that satisfies $D(w_i)\in 
\land V \otimes \land W(i)$ where $W(i)$ is spanned by the $w_j$ with $j < i$ (a K-S---for Koszul-Sullivan---basis).
The section $\sigma \colon B \to X$ induces a map $P \colon (\land V \otimes \land W, D) \to (\land V, d_B)$ which we may take to be the projection (cf. \cite[p.368]{FOT}).  
Notice that $P$ is a DG algebra map which implies that $  W \subset \land V \otimes \land W$ generates  a $D$-stable ideal. 
  
The relative minimal model construction extends the fundamental correspondence between homotopy classes of maps of rational spaces and DG homotopy classes of maps of minimal Sullivan algebras to the fibrewise setting.  Recall this correspondence in the ordinary case may be expressed in terms of a pair of  adjoint functors   (passage to Sullivan models and spatial realization) and implies a bijection
$[B'_\Q, B_\Q] \equiv [\B, \B']$
where,    the left-hand side is the   set of homotopy classes of maps between  rational nilpotent spaces and the right-hand side is the set of DG homotopy classes of maps between DG algebras    (cf. \cite{BG}).    

Now let  $X$ and $ X'$ be pointed fibrewise spaces, nilpotent as ordinary spaces,  over  a simply connected base  $B$.  Write  $[X', X]^B_B$ for the set of fibrewise pointed homotopy  classes of fibrewise pointed maps $X' \to X$ (see \cite[Sec.3]{CJ}).  Rationalizing, we obtain the corresponding set $[X'_\Q, X_\Q]^{B_\Q}_{B_\Q}$. 
 
 On the algebra side, let $(\land V, d_B) \to (\land V \otimes \land W, D)$ and $(\land V, d_B) \to (\land V \otimes \land W', D')$ denote the relative minimal models for $X$ and $X',$ respectively.     Let $\land(t, dt)$ denote the acyclic DG algebra with $|t| = 0$ and the $p_i \colon \land(t, dt) \to \Q$ the evaluations of $t$  at  $i = 0,1$. 
\begin{definition} \label{def:homotopic} 
Say two DG algebra maps $\psi_1, \psi_2 \colon \land V \otimes \land W \to \land V \otimes \land W'$ are {\em DG homotopic under  and over $ \land V, d_B$}, denoted by $\psi_1 \sim ^{\land V}_{\land V}\psi_2$, if there is a DG algebra map $$H \colon \land V  \otimes \land W \to \land V \otimes \land W' \otimes \land(t, dt)$$ satisfying $p_i \circ H = \psi_i$ with $H$ the identity on $\land V$ and $H$ commuting with the projections onto $\land V.$  
\end{definition} 
Let 
$[\land V \otimes \land W,  \land V \otimes \land W']_{\land V}^{\land V} $ denote the set  of equivalence classes of DG maps $\land V \otimes \land W, D \to \land V \otimes \land W', D'$ that are the identity on $\land V$ and commute with the projections,   under the relation of DG homotopy under and over $\land V$.    The following result is essentially a  special case of \cite[Pro.3.4]{FLS}. 
\begin{lemma}  \label{lem:o/u}   Suppose  $X$ and $X'$ are nilpotent as spaces and $B$ is   simply connected. Suppose either (i) $X'$ is a finite complex or (ii) the fibre $F$ for $X$ has only finitely many nonzero rational homotopy groups.  Then
there is a  natural bijection  
of sets $$ [\land V \otimes \land W,  \land V \otimes \land W']_{\land V}^{\land V} \equiv [X'_\Q, X_\Q]^{B_\Q}_{B_\Q}.$$
\end{lemma}
 \begin{proof}     
The bijective correspondence when $X'$ is finite is \cite[Pro.3.4]{FLS} in the case $Z = *.$ The assumption $X'$ finite is used in the proof  to perform 
an induction over the terms in a Moore-Postnikov decomposition of $X_\Q \to B_\Q$.  If $\pi_*(F) \otimes \Q$ is finite-dimensional this decomposition is finite already and so (ii) is sufficient for $X'$ infinite CW.  Naturality follows from the naturality of passage to models and spatial realization in the ordinary case. \end{proof}

We may  pass from the minimal model to an arbitrary model of $B$ in Lemma \ref{lem:o/u}.
Suppose $(\B, \delta_B)$ is a DG algebra such that there is a surjective quasi-isomorphism $\phi \colon (\land V, d_B) \to (\B, \delta_B)$.  In this case, we  say $(\B, \delta_B)$ is   a {\em model} for $B$. Forming the pushout gives a commutative square:
\begin{equation}  \label{eq:commsq}  
\xymatrix{  (\land V, d_B)  \ar[d]_{\phi}^{\simeq} \ar[rr]^I & & (\land V \otimes \land W, D) \ar[d]^{\phi \otimes 1}_\simeq \ar[r] & (\land W, d_F) \ar[d]^1 \\
(\B , \delta_B) \ar[rr]^{J} & & (\B \otimes \land W, D') \ar[r] & (\land W, d_F) }
\end{equation} 
in which the vertical maps are surjective quasi-isomorphisms (see \cite[p.66]{Ba}).  We refer to $J \colon (\B, \delta_B) \to (\B \otimes \land W, D')$
as a {\em relative model} for $p \colon X \to B.$  
As in Definition \ref{def:homotopic},  we say two DG algebra maps $\psi_1, \psi_2$ between relative models are  {\em DG homotopic under and over  $ (\B, \delta_B)$} if there is a DG algebra map $H \colon \B  \otimes \land W \to \B \otimes \land W' \otimes \land(t, dt)$ satisfying $p_i \circ H = \psi_i$ with $H$ the identity on $\B$ and $H$ commuting with the projections.   Let $$[\B \otimes \land W,  \B \otimes \land W']_{\B}^{\B}$$
denote the set of  equivalence classes, via DG homotopy under and over $\B$,  of DG algebra maps under  and over $\B$.   
We have:
\begin{lemma}  \label{lem:o/u2}    With hypotheses as in  Lemma \ref{lem:o/u},  there
is a  natural bijection  
of sets $$  [\B \otimes \land W,  \B \otimes \land W']_{\B}^{\B} \equiv [X'_\Q, X_\Q]^{B_\Q}_{B_\Q}.$$
\end{lemma}
\begin{proof} 
It suffices to produce a  natural bijection
$$[\B \otimes \land W,  \B \otimes \land W']_{\B}^{\B} \equiv [\land V \otimes \land W,  \land V \otimes \land W']_{\land V}^{\land V} $$ 
Consider the  diagram with vertical quasi-isomorphisms as in the middle column of (\ref{eq:commsq}). 
$$\xymatrix{ \land V \otimes \land W \ar[d]_{\phi \otimes 1}^{\simeq} \ar@{.>}[rr]^{\psi} && \land V \otimes \land W' \ar[d]^{\phi \otimes 1}_{\simeq} \\
\B \otimes \land W \ar@{.>}[rr]^{\overline{\psi}} &&  \B \otimes \land W'}
$$
A DG algebra map $\psi \colon \land V \otimes \land W \to \land V \otimes \land W'$ under and over   $\land V$ induces a DG map 
   $\overline{\psi} \colon \B \otimes \land W \to \B \otimes \land W'$ under and over $\B$ given by $\overline{\psi}(b) = b$ and $\overline{\psi}(w) = (\phi \otimes 1) \circ \psi(w)$.  Conversely, suppose given a DG algebra map  $\overline{\psi} \colon \B \otimes \land W \to \B \otimes \land W'$ under and over $\B$.  Apply  the relative lifting lemma \cite[Pro.14.6]{FHT} to  $\overline{\psi}\circ (\phi \otimes 1)$ to obtain  $\psi$, a DG algebra map    under and over $\land V$.   The proof that these assignments set up the needed bijection on homotopy sets  follows the standard argument (see \cite[Sec.14]{FHT} or \cite[Sec.3]{FLS}).  
\end{proof}

We now fix a pointed fibrewise space $X$ over a simply connected base $B$ and a  relative model written $(\B, d_B) \to (\B \otimes \land W, D)$. Write $F$ for the fibre so that $(\land W, d_F)$ is a minimal model for $F$.   A simple criterion for fibrewise rational triviality is   the following:

\begin{theorem} \label{homtrivial}   Let $X$ be a pointed fibrewise space  with 
 $X$ nilpotent and $B$ simply connected. Suppose either (i) $X$ is homotopy finite or (ii) $F$ has finitely many non-trivial rational homotopy groups. Suppose $D(w) =  d_F(w)$ for all $w \in W$ in the relative model for $X$.    Then $X$ is rationally fibrewise trivial. 
\end{theorem}
\begin{proof}
The hypothesis on $D$ implies the inclusion $j \colon (\land W, d_F) \to (\B \otimes \land W, D)$ is a DG algebra map.    Thus  $$1 \otimes j   \colon (\B, d_B) \otimes (\land W, d_F) \to  
(\B \otimes \land W, D)$$ is a DG equivalence under and over  $\B.$  The spatial realization of $1\otimes j$  provided by Lemma \ref{lem:o/u2}
is the needed fibrewise   equivalence $  X_\Q  \simeq^{B_\Q}_{B_\Q} B_\Q \times F_\Q$.    
\end{proof}

Now suppose $X$ is a fibrewise H-space with fibrewise multiplication $m$ and  fibre $G$. We describe the Sullivan model for   $m$.   A relative model for the fibre product  $X \times_B X$ is given by $$\B \otimes \land W \otimes \land W  \cong   (\B \otimes \land W) \otimes_\B (\B \otimes \land W)$$ with $d_G(w) = d_G(w') = 0$ for $ w \in W$ (cf. \cite[Pro.15.8]{FHT}). Here we write $w = 1 \otimes w \otimes 1$ and $w' = 1 \otimes 1 \otimes w$ .  The model for $m$ is 
 then  a map $$C \colon \B \otimes \land W \to \B \otimes \land W \otimes \land W $$ with  $C$ the identity on $\B$ and $C$ commuting with the projections.  Since $m \circ (1 \times \sigma) \circ \Delta$ and $m \circ (\sigma \times 1) \circ \Delta$ are pointed fibrewise homotopic to the identity we obtain that 
$C(w) - w - w' \subset 
\B \otimes  \land^+ W \otimes \land^{+} W.$ 
 Thus we may write
$$C(w) = C_0(w) + C_2(w) + C_3(w) + \cdots$$
for $w \in W$  where $C_0(w) = w + w'$ and $C_r(w) \in \B \otimes
\sum_{s=1}^{s=r-1}( \land^s W \otimes \land^{r-s} W)$ for $r \geq 2,$ 
as in the introduction.  Here $P' \circ C_r(w) = 0$ where $P' \colon \B \otimes \land W \otimes \land W\to \B$ is the projection. 
We  refer to a map $C$ satisfying these conditions   as a {\em fibrewise multiplication of models}.

There is a natural notion of equivalence, or isomorphism, of
fibrewise multiplications on a fibrewise H-space.  Suppose $m_1
\colon X \times_B X \to X$ and $m_2
\colon Y \times_B Y \to Y$ are two such.  We say that $m_1$ and
$m_2$ are \emph{fibrewise equivalent} fibrewise multiplications if
there is a fibrewise pointed, fibrewise homotopy equivalence $f \colon X \to Y$
for which the diagram
$$\xymatrix{X \times_B X \ar[d]_{f\times_B f} \ar[r]^-{m_1} & X \ar[d]^{f}\\
Y \times_B Y \ar[r]_-{m_2}& Y }$$
is pointed fibrewise homotopy commutative.

This translates  to give a corresponding notion of
equivalence between  fibrewise multiplications of models.  We phrase this in the form in which we use it below.

 \begin{definition}\label{def: equivalence of model multn}
Say that two fibrewise multiplications of models
$$C_i\colon \big(\B \otimes \land W, D_i\big) \to \big(\B \otimes \land W \otimes \land
W, D_i\big),$$
for $i = 1,2$,  are \emph{equivalent} if there exists a DG algebra isomorphism $\phi \colon
\big(\B \otimes \land W, D_2\big) \to \big(\B \otimes \land W, D_1\big)$ that is the
identity on $\B$ and commutes with the projections, such that
\begin{equation}
\label{eq:C}
\xymatrix{\B \otimes \land W \ar[d]_{\phi} \ar[r]^-{C_2}&\B \otimes \land W \otimes \land
W \ar[d]^{1 \otimes \phi\otimes\phi}  \\
\B \otimes \land W \ar[r]_-{C_1} &\B \otimes \land W
\otimes \land W }
\end{equation}
commutes  up to DG homotopy under and over  $\B$.  We will denote this equivalence by $(D_1, C_1) \equiv_\B^\B (D_2, C_2)$ or, in case the differential is the same on either model, simply by $C_1 \equiv_\B^\B C_2$.
\end{definition}

Our criterion  for rational  fibrewise H-triviality is the following:

\begin{theorem}  \label{eqtrivial}  Let $X$ be a fibrewise H-space with X a nilpotent space, $B$ simply connected and either $X$ or the fibre $G$   homotopy finite.   Suppose that   $(D, C) \equiv_\B^\B (D=d_B\otimes 1, C_0)$.    Then $X$ is rationally fibrewise H-trivial.
\end{theorem}
\begin{proof}
 Let $  m_0 \colon X_{0} \times_{B_\Q} X_0 \to X_0$ denote the  spatial realization of the DG algebra map $C_0 \colon \B \otimes \land W \to \B \otimes \land W \otimes \land W$ as given by Lemma \ref{lem:o/u2}.  Then $X_0 \simeq^{B_\Q}_{B_\Q} B_\Q \times G_\Q$.    By naturality, $m_0 = 1  \times_{B_\Q}  (m_0) _{ G_\Q}$ where $(m_0)_{G_\Q}$ is the restriction of $m_0$ to $G_\Q$.     Applying Lemma \ref{lem:o/u2} again,   the spatial realization of the diagram (\ref{eq:C}) for $C_1 = C_0$ and $C_2 = C$ yields  the needed  $H$-equivalence $X_\Q, m_\Q \simeq^{B_\Q}_{B_\Q} X_0, m_0.$    \end{proof}

\section{Fibrewise Hopf and Leray-Samelson Theorems}\label{sec:LS}

Hopf's theorem, as enunciated at the start of this paper, may be viewed as the statement that the differential vanishes in the Sullivan minimal model of an  $H$-space.  This  result  may be proved directly in the framework of  Sullivan models (see \cite[Ex.3, p.143]{FHT}).  Our first main result in this section generalizes this form of Hopf's theorem to fibrewise H-spaces that satisfy a certain technical condition.  

We first introduce some  notational conventions for describing elements of $\B \otimes \land W \otimes \land W$. 
Suppose $\{w_i\}$ is a basis of $W$.   For $r \geq 2$, let  $I = (i_1, i_2, \dots , i_r)$ be an $r$-tuple of indices, in which $i_1  \leq i_2 \leq \cdots \leq i_r$.  We denote the length $r$  of an $r$-tuple $I$ by $|I|$.  We will write $w_I$ for
$w_{i_1} w_{i_2} \cdots w_{i_r} \in \land W$.  Then let $\varepsilon =
(\epsilon_1, \epsilon_2, \dots , \epsilon_r)$ be a binary $r$-tuple with each $\epsilon_i
= 0$ or $1$, and such that $1 \leq \sum_{i=1}^{r} \epsilon_i \leq r-1$ (which excludes the $r$-tuples $(0, 0, \dots, 0)$ and $(1, 1,\dots, 1)$).   We will write $w^{\epsilon}_I$ for
$w_{i_1}^{\epsilon_1} w_{i_2}^{\epsilon_2} \cdots w_{i_r}^{\epsilon_r} \in \land W \otimes \land W$, with each $w_{i}^0 = w_i \in
\land W \otimes1$ and $w_{i}^1 = w_i' \in
1\otimes\land W $.  We allow for repeated indices in $I$ if even degree generators in $W$ are present.  Also, we write $w_I$ for $w^{(0, 0, \dots, 0)}_I = w_{i_1}w_{i_2}\cdots
w_{i_r} \in \land W \otimes1$ and $w'_I$ for $w^{(1, 1, \dots, 1)}_I =
w'_{i_1}w'_{i_2}\cdots w'_{i_r}\in
1\otimes\land W$.  Then we will write the typical element of $\B \otimes \land^{+} W \otimes \land^{+} W$ whose terms involve $r$-fold products of exactly those $w_i$'s whose indices appear in $I$ as
$$\chi_I = \sum_{\varepsilon} b^{\varepsilon}_I w^{\varepsilon}_I \in \B \otimes \land^{+} W \otimes \land^{+} W, $$
with each $b^{\epsilon}_I \in \B$ and where the sum is over all binary $r$-tuples $\varepsilon =
(\epsilon_1, \epsilon_2, \dots , \epsilon_r)$ with each $\epsilon_i
= 0$ or $1$, excluding the $r$-tuples $(0, 0, \dots, 0)$ and $(1, 1,
\dots, 1)$.   

Finally,  we write $S_i$ for the
binomial $w_i + w'_i$ and extend this to $$S_I = S_{i_1}S_{i_2}\cdots
S_{i_r} = (w_{i_1}+w'_{i_1})\cdots (w_{i_r} + w'_{i_r}).$$  These
notations are used throughout the remainder of this section.

\begin{theorem} \label{D=0}
    Let  $(\B, d_B) \to (\B \otimes \land W, D)$ be a relative model     with the projection $\B \otimes \land W \to \B$  a DG algebra map admitting  
   $C \colon \B \otimes \land W \to \B \otimes \land W \otimes \land W$   a fibrewise multiplication of models.  Suppose  $W$ and $\B$ are of finite type and   that we have  $D(W) \subseteq \B \otimes
\land^{\geq 2} W$.  Then there is an equivalence of fibrewise multiplications of models $(D,C) \equiv^\B_\B (d_B\otimes 1, C')$ for some fibrewise multiplication  of models  $C'$.    \end{theorem}

\begin{proof}
Suppose $\{w_i\}$ is a K-S basis of $W$, ordered so as to refine the
partial ordering by degree.  That is, we suppose that $|w_i| < |w_j|$ implies $i < j$---such a choice is always possible.  We will proceed by induction over the index set, and show an equivalence $(D, C)\equiv^\B_\B (D^p, C^p)$, for each $p \geq 1$, with $D^p(w_i) = 0$ for all $i \leq p$.     

Induction starts with $D^1 = D$ and $C^1 = C$, since we have $D(w_1) = 0$ from our basic assumptions.

Now suppose inductively that, for some $k \geq 2$, we have $(D, C)\equiv^\B_\B (D^{k-1}, C^{k-1})$, with $D^{k-1}(w_i) = 0$ for all $i <k$.      By hypothesis,  $D^{k-1}(w_k)$ may be written as
$$D^{k-1}(w_k) = D^{k-1}_r(w_k) + D^{k-1}_{r+1}(w_k) + \cdots,$$
for some $r\geq2$ and each  $D^{k-1}_j(w_k) \in \mathcal{B}\otimes \land^j W$.        
Using the notation introduced above, we may write
\begin{equation}\label{eq: Dr}
D^{k-1}_{r}(w_k) = \sum_{|I| = r}\, b_I w_I = \sum_I\, b_I w_{i_1}w_{i_2}\cdots
w_{i_{r}},
\end{equation}
where the sum is over all $r$-tuples $I = (i_1, i_2, \dots, i_{r})$ with $i_1  \leq i_2 \leq \cdots \leq i_r$,
and for which there are corresponding elements $b_I \in \mathcal{B}$ for which $|b_I w_I| = |w_k| + 1$.  

Recall that $C^{k-1} \colon \B \otimes \land W \to \B \otimes \land W \otimes \land W$ denotes the fibrewise multiplication of models that we have in hand.  Modulo terms in $\B \otimes \land^{+} W \otimes \land^{+} W$,  we have $C^{k-1}(w_i) \equiv w_i + w_i'$ for
each $w_i \in W$.  Therefore, modulo terms in $\sum_{i+j\geq r+1}\B \otimes \land^i W \otimes \land^j W$, we have
\begin{equation}\label{eq: CD}
C^{k-1} D^{k-1} (w_k) \equiv C_0 D^{k-1}_{r}(w_k) = \sum_I\, b_I (w_{i_1}+ w'_{i_1})(w_{i_2}+w'_{i_2})\cdots
(w_{i_{r}}+ w'_{i_{r}}).
\end{equation}
Next, we may write
$$C^{k-1}(w_k) = w_k + w'_k + \sum_{|J| \geq 2, \varepsilon} \beta^{\varepsilon}_{J} w^{\varepsilon}_{J}$$
where, as in the notation introduced above, for each $s\geq2$, the sum is over all $s$-tuples $J = (j_1, \cdots, j_s)$ with $j_1 \leq \cdots \leq j_s$, and binary $s$-tuples $\varepsilon = (\epsilon_1, \cdots, \epsilon_s)$ such that $1 \leq \sum_{i=1}^{r} \epsilon_i \leq r-1$, for which there are corresponding elements $\beta^{\varepsilon}_J \in \B$ such that $|\beta^{\varepsilon}_Jw^{\varepsilon}_J| = |w_k|$.    For degree reasons, the only indices $j_i$ that may occur in the $s$-tuples in this sum satisfy $j_i < k$.  By our induction hypothesis, we have $D^{k-1}(w_{j_i}) = 0$ for all such terms and thus, by applying $D^{k-1}$ to the above displayed equation, and working modulo terms in 
$\sum_{i+j\geq r+1}\B \otimes \land^i W \otimes \land^j W$, we have
\begin{equation}\label{eq: DC}
D^{k-1} C^{k-1} (w_k) \equiv D^{k-1}_rw_k + D^{k-1}_rw'_k +  \sum_{2 \leq |J| \leq r, \varepsilon} d_B(\beta^{\varepsilon}_{J}) w^{\varepsilon}_{J}.
\end{equation}
In this last expression, we have abused notation somewhat, and written $D^{k-1}$ (hence $D^{k-1}_r$) for the differential $D^{k-1}\otimes_\B 1 + 1 \otimes_\B D^{k-1}$ in $\B \otimes \land W \otimes \land W$.

Now, using the fact that $C^{k-1}$ is a DG map, equate terms that occur in $(\ref{eq: CD})$ and $(\ref{eq: DC})$ (see (\ref{eq: Dr})) from $\sum_{i+j = r}\B \otimes \land^i W \otimes \land^j W$ , and obtain the following identity 
\begin{equation}\label{eq: ID}
\sum_{|I|=r}\, b_I (w_{i_1}+ w'_{i_1})\cdots
(w_{i_{r}}+ w'_{i_{r}}) = \sum_{|I|=r}\, b_I w_I + \sum_{|I|=r}\, b_I w'_I +  \sum_{|J| = r} d_B(\beta^{\varepsilon}_{J}) w^{\varepsilon}_{J}.
\end{equation}
From this identity, if we equate, for example, coefficients from either side in the element
$w_I^{(1, 0, \cdots, 0)}$, for each $r$-tuple $I$, then we obtain that each $b_I$ is a boundary,
$$b_I = d_B \eta_I,$$
for some $\eta_I \in \B$, for each $I$.  There are two possibilities for the $\eta_I$.  If the index $i_1$ is not repeated in $I$, so that $i_1<i_2$, then we have unique contributions in $w_I^{(1, 0, \cdots, 0)}$ on either side of (\ref{eq: ID}), so $b_I = d_B( \beta_I^{(1, 0, \cdots, 0)})$ and we set $\eta_I = \beta_I^{(1, 0, \cdots, 0)}$.   If $|w_1|$ is even, and $i_1$ is repeated $N$-times, say, in $I$ with $N \leq r$, then $I = (i_1, \cdots, i_1, i_{N+1}, \cdots, i_r)$ with $i_1 < i_{N+1}$.  Here, the left-hand side of  (\ref{eq: ID}) contributes $N b_I$ in $w_I^{(1, 0, \cdots, 0)}$, whereas the right-hand side contributes a sum $\sum_{\varepsilon} d_B(\beta^\varepsilon_I)$, with the sum over those $\varepsilon$ with a single $1$ in one of the first $N$ places, and zeros elsewhere.  Then here we set
$$\eta_I = \frac{1}{N} \sum_{\varepsilon} \beta^\varepsilon_I.$$
Now using once more our induction hypothesis that $D^{k-1}(w_{i_j}) = 0$ for each basis element $w_{i_j}$ that appears here, we may write
$$D^{k-1}_r(w_k) = \sum_{|I| = r} d_B(\eta_I) w_I = D^{k-1}\big( \sum_{|I| = r} \eta_I w_I\big).$$
Now make a change of generators, by which we mean the following.  Define a map
$$\Phi \colon \B \otimes \land W \to \B \otimes \land W$$
by setting $\Phi = \mathrm{id}$ on $\B$ and on all basis elements $w_i$ of $W$ other than $w_k$.  On $w_k$, define
$$\Phi(w_k) = w_k -  \sum_{|I| = r} \eta_I w_I.$$
It is obvious that $\Phi$ defines an isomorphism.  Then define a new differential by setting
$$D^{k-1, r+1} = \Phi^{-1} \circ D^{k-1} \circ \Phi$$
on $\B \otimes \land W$, so that $\Phi$ becomes a DG isomorphism
$$\Phi \colon \big(\B \otimes \land W, D^{k-1, r+1}\big) \to \big(\B \otimes \land W, D^{k-1}\big).$$
This DG isomorphism is evidently both under and over $\B$.  Furthermore, we also define a new fibrewise multiplication of models $C^{k-1, r+1}$, by setting 
$$C^{k-1, r+1} = (\Phi^{-1}\otimes_\B \Phi^{-1})\circ C^{k-1} \circ \Phi.$$
One readily checks that $C^{k-1, r+1}$ commutes with differentials, and so  is indeed a fibrewise multiplication.  Thus  $\Phi$ is an equivalence 
$$(D^{k-1}, C^{k-1}) \equiv^\B_\B (D^{k-1, r+1}, C^{k-1, r+1}).$$
But from our construction, we have that $D^{k-1, r+1} = D^{k-1} = 0$ on basis elements $w_i$ with $i <k$, and that  $D^{k-1, r+1}(w_k) \in \B \otimes \land^{\geq r+1} W$.
By induction, we have an equivalence $(D^{k-1}, C^{k-1}) \equiv^\B_\B (D^{k}, C^{k})$, with $D^k(w_i) = 0$ for $i \leq k$.

For finite-dimensional $W$, the theorem follows immediately by induction over the (finitely many) generators.  In case $W$ is infinite-dimensional, note that in the above induction step, in which we built an equivalence $(D^{k-1}, C^{k-1}) \equiv^\B_\B (D^{k}, C^{k})$, the  change of generators isomorphism is the identity on generators $w_i$ with $i < k$.  Since our equivalences, hence the differentials and the fibrewise multiplications of models,  ``stabilize" in this way, we obtain the same conclusion even when $W$ is infinite-dimensional. \end{proof}

As a consequence we obtain the following result which represents a fibrewise version of Hopf's theorem.  

\begin{theorem}\label{Hopf}
Let  $(\B, d_B) \to (\B \otimes \land W, D)$ be a relative model  with projection a DG algebra map admitting a fibrewise multiplication $C.$  Suppose   $W$ and $\B$ are of finite type satisfiying (1)  $\B$ has zero odd-dimensional cohomology and (2) $W$ is  zero  in even dimensions.   Then there is an equivalence of fibrewise multiplications of models $(D,C) \equiv^\B_\B (d_B\otimes 1, C')$ for some  $C'$. 
\end{theorem}

\begin{proof} 
We show first that we may adjust the relative model so as to satisfy the hypotheses of \thmref{D=0}, and then make our conclusion from that result.  Let $C$ denote the fibrewise multiplication of models induced by the multiplication for $X$. 
Suppose $\{w_i\}$ is a K-S basis of $W$; we argue by induction over $i$.  Induction starts with $i=1$, where we have $D(w_1) = 0$.

Now suppose inductively that we have $D(w_i) \in \B \otimes
\land^{\geq 2} W$ for all $i < k$.  We may write 
$$D(w_k) = \sum_{i<k} b_{i,k} w_i  + \chi,$$
for $\chi \in \B \otimes \land^{\geq 2} W$.  Then $0 = D^2(w_k) = \sum_{i<k} d_B(b_{i,k} )w_i  - \sum_{i<k} b_{i,k} D(w_i) + D(\chi)$.  By our induction hypothesis, and the fact that $\B \otimes
\land^{\geq 2} W$ is $D$-stable, it follows that $d_B(b_{i,k}) = 0$ for each $i$.  Then, since $w_k$ and each $w_i$ are of odd degree, each $b_{i, k}$ is an odd-degree cycle of $B$, hence exact: $b_{i,k} = d_B(\eta_{i,k})$ for each $i$ and $k$, some $\eta_{i,k} \in B$.  

Define a change of generators  isomorphism 
$$\Psi \colon \B \otimes \land W \to \B \otimes \land W$$
by setting $\Psi = \mathrm{id}$ on $\B$ and on all basis elements $w_i$ of $W$ other than $w_k$.  On $w_k$, define
$$\Psi(w_k) = w_k -  \sum_{i<k} \eta_{i,k} w_i.$$
Then define $D' = \Psi^{-1} \circ D \circ \Psi$ and 
$$C' = (\Psi^{-1}\otimes_\B \Psi^{-1})\circ C \circ \Psi,$$
so that $\Psi$ gives an equivalence $(C, D) \equiv^\B_\B (C', D')$.    
From the construction of $\Psi$, we have that $D' = D$ on basis elements $w_i$ with $i <k$, and that 
$D'(w_k) \in \B \otimes \land^{\geq 2} W$. It  follows immediately by induction that, for finite-dimensional $W$,  we have an equivalence $(C, D) \equiv^\B_\B (C', D')$ such that $D'$ satisfies the hypotheses of \thmref{D=0}.  If $W$ is infinite-dimensional, then the same conclusion holds as the change of generators used above is the identity on the $w_i$ with $i < k$, which means that the differential and the multiplication of models stabilize as we proceed over the K-S basis of $W$.  \end{proof}
We can now deduce Theorem \ref{main}. 
\begin{proof}[Proof of Theorem \ref{main}]
Let $X$ be a fibrewise   H-space satisfying the conditions in Theorem \ref{main}.   Then the relative   model $\B \to \B \otimes \land W$ for $X \to B$  satisfies the conditions of Theorem \ref{Hopf}. Furthermore, $X, B$ and $G$ satisfy the hypotheses of Theorem \ref{homtrivial}.  Theorems \ref{Hopf} and  \ref{homtrivial} thus imply $X$ is rationally fibrewise trivial.  \end{proof}

We observe that an arbitrary fibrewise H-space need not have a Sullivan model that satisfies the conclusion of \thmref{Hopf}.  Indeed, if either of the hypotheses of \thmref{Hopf} is omitted, then the conclusion may fail, as the following examples illustrate.

\begin{example}\label{ex: non-Hopf}
Let  $X$ be a (rational) fibrewise space with minimal model of the form $(\land (b_3), 0) \to (\land (b_3) \otimes \land (w_3, w_5), D, C)$, where subscripts denote the degrees of generators, $D$ is defined on generators by $D(b_3) = D(w_3) = 0$ and $D(w_5) = b_3 w_3$, and $C = C_0$.  Here we have $H^3(B;\Q) \not=0$, and $X$ clearly fails to satisfy the conclusion of \thmref{Hopf}.  Note that $X$ may be described as the principal $K(\Z, 5)$-fibration over $S^3 \times S^3$ obtained by pulling back the principal $K(\Z, 5)$-fibration over $K(\Z, 6)$, over the map given by a fundamental class in $H^6(S^3 \times S^3)$.  Up to rational homotopy, then, we may view $X$ as a fibrewise space over $S^3$ with fibre $S^3 \times K(\Z, 5)$.  

Next, consider the free loop space on $S^2$, which we denote by $\Lambda S^2$.  This is a fibrewise H-space over $S^2$, and has relative model
$$(\land(x, y),  d_B) \to (\land(x,y)\otimes \land(\overline{x}, \overline{y}), D)$$
in which $|x| = |\overline{y}| = 2$, $|y| = 3$, $|\overline{x}| = 1$ and the differential is given on generators by $d_B(x) = 0$, $d_B(y) = x^2$, $D(\overline{x}) = 0$, and $D(\overline{y}) = -2 x \overline{x}$ (\cite[Ex.5.12]{FOT}).  In this case, one has the standard fibrewise multiplication of models defined by $C_0(\overline{x}) = \overline{x} + \overline{x}'$ and $C_0(\overline{y}) = \overline{y} + \overline{y}'$.  But here one also has a non-equivalent fibrewise multiplication of models, defined by $C(\overline{x}) = \overline{x} + \overline{x}'$ and $C_0(\overline{y}) = \overline{y} + \overline{y}' + \overline{x}\,\overline{x}'$.  From the calculations in \cite[Sec.4.4]{FTV}, it follows that this latter is the fibrewise multiplication of models determined by the fibrewise H-structure of the free loop space.  In this case, we have $H^*(B;\Q)$ zero in odd degrees, but the fibre $\Omega S^2$ has a non-zero rational homotopy group in degree $2$.  Clearly $\Lambda S^2$ fails to satisfy the conclusion of  \thmref{Hopf}.
\end{example}
 
We now turn to the fibrewise extension of the Leray-Samelson Theorem.
We assume the fibrewise  multiplication $m \colon X \times_B X \to X$ is (pointed-fibrewise, homotopy) associative. This means that the diagram
$$\xymatrix{X \times_B X \times_B X \ar[r]^-{m \times_B 1}
\ar[d]_-{1 \times_B m}& X \times_B X \ar[d]^{m} \\
X \times_B X \ar[r]_-{m} & X}$$
is pointed-fibrewise homotopy commutative.  The diagram above translates into a diagram
\begin{equation} \label{assoc} \xymatrix{\B \otimes \land W \ar[r]^-{C}
\ar[d]_-{C}& \B \otimes \land W \otimes \land
W \ar[d]^{C \otimes_{\B} 1} \\
\B \otimes \land W \otimes \land W \ar[r]_-{1 \otimes_{\B}
C} & \B \otimes \land W \otimes \land W\otimes \land W} \end{equation}
which  commutes up to DG homotopy under and over  $\B$.  

Suppose now that $H(\B)$ is evenly graded and $W$ is oddly graded, as in the hypotheses of \thmref{Hopf}.   Then $D(W) =0$ in the relative Sullivan model $ \B \to \B  \otimes \land W$ (\thmref{Hopf}) and hence also $D(\land
W \otimes \land W\otimes \land W) = 0$ in $\B \otimes \land
W \otimes \land W\otimes \land W$.

The following is   our main technical result.  It represents a fibrewise extension of the Leray-Samelson Theorem.

\begin{theorem}\label{LS}
 Let  $(\B, d_B) \to (\B \otimes \land W, D)$ be a relative model  with projection a DG algebra map  with  a given fibrewise multiplication $C.$  Suppose $W$ and $\B$ are of finite type satisfying (1)  $\B$ has zero odd-dimensional cohomology and (2) $W$ is  zero  in even dimensions.   If $C$ is associative,  then there is an equivalence of  fibrewise multiplications of models $(D, C) \equiv^\B_\B (d_B\otimes 1, C_0).$  Here $C_0$ denotes the standard multiplication defined by $C_0(w) = w + w'$ for each $w \in W$.
\end{theorem}

\begin{proof}
By \thmref{Hopf}, we may assume that $D = d_B\otimes 1$.   
As in previous arguments, we proceed by induction over a K-S basis 
$\{w_i\}$ of $W$, and show that, for each $k\geq 1$, we have an equivalence $(d_B\otimes 1, C) \equiv^\B_\B (d_B\otimes 1, C^k)$, with $C^k(w_i) = C_0(w_i) = w_i + w'_i$ for $i \leq k$.  

For degree reasons, we must have $C(w_1) = C_0(w_1) = w_1 + w'_1$, and so setting $C^1 = C$ starts our induction.  

Now suppose inductively that we have an equivalence of  fibrewise multiplications of models  $(D, C) \equiv^\B_\B (D, C^{k-1})$ for some $k\geq 2$, with $D= d_B\otimes1$ (and so $D(w_i) = 0$ for all $i$) and $C^{k-1}(w_i) = C_0(w_i) = w_i + w'_i$ for all $i < k$.  
First write 
$$C^{k-1}(w_k) = w_k + w'_k + P_2 + P_3 + \cdots,$$
with each $P_r \in \mathcal{B}\otimes \sum_{s=1}^{s=r-1}( \land^s W
\otimes \land^{r-s} W)$ of odd degree.  Since $0 = C^{k-1}(Dw_k) = D\big(C^{k-1}(w_k)\big)$, and $D(W) = 0$, it follows that each $P_r$ is a $D$-cycle.  
We will show that $C^{k-1}$ may be successively adjusted so as to remove each term $P_r$.  For this, assume inductively that for some $r \geq 2$ we have an equivalence of fibrewise multiplications  of models
$(D, C^{k-1}) \equiv^\B_\B (D, C^{k-1}_r)$, such that $C^{k-1}_r(w_i) = C^{k-1}(w_i) = C_0(w_i) = w_i + w'_i$ for all $i < k$, and  
$$C^{k-1}_r(w_k) = w_k + w'_k + P_r + P_{r+1} + \cdots$$
There are two cases, which we handle slightly differently.

\noindent\textit{Case I. $r = 2m$ is even.}
We may write
$$P_{2m} = \sum_{|I| = 2m, \varepsilon} \beta^{\varepsilon}_I w^{\varepsilon}_I,$$
for $\beta^{\varepsilon}_I \in \B$.  Now $D(P_{2m}) = 0$ implies that each $d_B(\beta^{\varepsilon}_I) = 0$, and since $P_{2m}$ and each $w_i$ for $i <k$ is of odd degree, it follows that each $\beta^{\varepsilon}_I$ is an odd-degree cycle in $\B$.  Our assumption on the homology of $\B$ now gives that each $\beta^{\varepsilon}_I = d_B(\eta^{\varepsilon}_I)$ for some $\eta^{\varepsilon}_I \in \B$, and hence we have  
$$P_{2m} = D\big(\sum_{|I| = 2m, \varepsilon} \eta^{\varepsilon}_I w^{\varepsilon}_I\big).$$
For brevity, write this last term as $D(\eta_{2m})$.  Define a fibrewise multiplication of models 
$C^{k-1}_{r+1} \colon \mathcal{B} \otimes \land W \to \mathcal{B} \otimes \land W\otimes \land W$ by 
$C^{k-1}_{r+1}(w_i) = C^{k-1}_{r}(w_i)$ for all $i \not= k$, and $C^{k-1}_{r+1}(w_k) = w_k + w'_k +  P_{r+1} + \cdots$.  Then define a DG homotopy over and under $\B$ 
$$H_{2m} \colon \B \otimes \land W \to \B \otimes \land W\otimes \land W\otimes (t, dt)$$
by setting $H_{2m} = \mathrm{id}$ on $\B$,  $H_{2m}(w_i) = C^{k-1}_r(w_i)$ for $i \not= k$, and 
$$H_{2m}(w_k) = C^{k-1}_r(w_k) - P_{2m} t - \eta_{2m} dt.$$
Using the fact that $D(w_i) = 0$ for all $i$, we readily check that $H_{2m}$ so defined is a DG map.  Clearly, this gives a DG homotopy $C^{k-1}_r\sim ^\B_\B C^{k-1}_{r+1}$.  Hence, the identity gives an equivalence of fibrewise multiplications of models 
$(D, C^{k-1}_r) \equiv^\B_\B (D, C^{k-1}_{r+1})$.

\noindent\textit{Case II. $r = 2m+1$ is odd.}
We may write
$$P_{2m+1} = \sum_{|I| = 2m+1, \varepsilon} \beta^{\varepsilon}_I w^{\varepsilon}_I,$$
for $\beta^{\varepsilon}_I \in \B$.  Arguing as before, we have that each $\beta^{\varepsilon}_I$ is a cycle in $\B$, but now of even degree, and thus not necessarily exact.   To handle this situation, we choose a direct sum decomposition $E \oplus N \cong \mathcal{Z}(\B)$ of the cycles (in all degrees) of $\B$, with $E = d_B(\B)$ and $N$ a complement.  As a vector space, therefore, $N \cong H^*(\B)$.  Now refine the sum above, to write
\begin{equation}\label{eq:P3}
P_{2m+1} = \sum_{|I| = 2m+1, \varepsilon} \alpha^{\varepsilon}_I w^{\varepsilon}_I + \sum_{|I| = 2m+1, \varepsilon} \gamma^{\varepsilon}_I w^{\varepsilon}_I,
\end{equation}
with each $\alpha^{\varepsilon}_I = d_B(\eta^{\varepsilon}_I)$ for some $\eta^{\varepsilon}_I \in \B$, and each $\gamma^{\varepsilon}_I  \in N \subseteq \mathcal{Z}(\B)$.  As in the previous step, for brevity we write the first of these sums as $D(\eta_{2m+1})$.  
Now we use the  associativity of $C$ (which entails the associativity of any equivalent fibrewise multiplication of models).  Let $K$ be a DG homotopy over and under $\B$ that makes (\ref{assoc}) DG homotopy commute (with $C^{k-1}_r$ replacing $C$ there).
Then we have 
$$K(w_k) = (1\otimes_\B C^{k-1}_r)\circ C^{k-1}_r(w_k) + \sum_{i\geq 1} A_i t^i + \sum_{j\geq 0} B_j t^j dt,$$
for elements $A_i, B_j \in \B\otimes \land W\otimes \land W\otimes \land W$, such that substituting $t = 1$ and $dt = 0$ yields $(C^{k-1}_r\otimes_\B 1)\circ C^{k-1}_r(w_k)$.   Equating $0 = K(Dw_k) = D\big(K(w_k)\big)$, together with the fact that $0 = (1\otimes_\B C^{k-1}_r)\circ C^{k-1}_r(Dw_k) = D\big((1\otimes_\B C^{k-1}_r)\circ C^{k-1}_r(w_k)\big)$, yields the identities $D(A_p) = 0$ and $pA_p = D(B_{p-1})$, for each $p \geq 1$.  Thus we have 
\begin{equation}\label{eq:assoc}
(C^{k-1}_r\otimes_\B 1)\circ C^{k-1}_r(w_k) = (1\otimes_\B C^{k-1}_r)\circ C^{k-1}_r(w_k) + D\big(\sum_{i\geq 1} \frac{1}{i} B_{i-1}\big).\end{equation}
Note that the last term appearing here is in $d_B(\B)\otimes \land W\otimes \land W\otimes \land W$.  
From (\ref{eq:P3}), denote the two sums by $P'_{2m+1} = \sum_{|I| = {2m+1}, \varepsilon} \alpha^{\varepsilon}_I w^{\varepsilon}_I = D(\eta_{2m+1})$ and $P''_{2m+1} = \sum_{|I| = {2m+1}, \varepsilon} \gamma^{\varepsilon}_I w^{\varepsilon}_I$.   
If we equate terms  from 
$\mathcal{B}\otimes \sum_{l+s+n = {2m+1}}( \land^l W \otimes \land^{s}
W\otimes \land^{n} W)$ that appear in (\ref{eq:assoc}),
then we see that  $(C^{k-1}_r\otimes_\B 1)(P'_{2m+1})$ and $(1\otimes_\B C^{k-1}_r)(P'_{2m+1})$  contribute terms from $d_B(\B)\otimes \land W\otimes \land W\otimes \land W$, whereas $(C^{k-1}_r\otimes_\B 1)(P''_{2m+1})$ and $(1\otimes_\B C^{k-1}_r)(P''_{2m+1})$  contribute terms from $N\otimes \land W\otimes \land W\otimes \land W$.  Since $N$ and $d_B(\B)$ are complementary in $\mathcal{Z}(\B)$, we may equate these contributions separately, and we find that we have the  identity
\begin{equation}\label{eq:basic identity}
\alpha(P''_{2m+1}) + \beta(P''_{2m+1}) = \gamma(P''_{2m+1}) + \delta(P''_{2m+1}).
\end{equation}
Here, we have used the notations $\alpha, \beta, \gamma, \delta$ to
denote algebra maps $\mathcal{B}\otimes \land W \otimes \land W\to
\mathcal{B}\otimes \land W \otimes \land W\otimes \land W$, each
defined as the identity on $\mathcal{B}$ and extended on generators
$w$ and $w'$ of $\land W \otimes \land W$ by: $\alpha(w) = w$ and
$\alpha(w') = w'$; $\beta(w) = w + w'$ and $\beta(w') = w''$,
$\gamma(w) = w$ and $\gamma(w') = w' + w''$; $\delta(w) = w'$ and
$\delta(w') = w''$.  From \propref{prop:alpha beta gamma delta}
below, we have that $P''_{2m+1}$ is of the form
\begin{equation}\label{eq:basic form}
P''_{2m+1} = \sum_I b_I (S_I - w_I - w'_I),
\end{equation}
where the sum is over $({2m+1})$-tuples $I = (i_1, i_2, \ldots,  i_{2m+1})$ with $i_1 <
i_2< \cdots < i_{2m+1}$ and $b_I \in \B$.    
Now use this expression to define an isomorphism
$\Psi$ of $\mathcal{B}\otimes \land W$ as
the identity on $\mathcal{B}$ and all generators of $\land W$ other than $w_k$,  and 
$$\Psi(w_k) = w_k + \sum_I b_I w_I.$$
Use this to define a (tautologically equivalent) fibrewise multiplication of models to $\overline{C}^{k-1}_{r}$, as $\overline{C}^{k-1}_{r} =
(\Psi\otimes_{\mathcal{B}}\Psi)\circ C^{k-1}_{r}\circ \Psi^{-1}$.  One easily checks  that
$\overline{C}^{k-1}_r$ agrees with $C^{k-1}_{r}$ on generators $w_i$ with $i < k$, and on
$w_k$ takes the form
$$\overline{C}^{k-1}_r(w_k) = w_k + w'_k + P'_{2m+1} + P_{r+1} + \cdots.$$
Finally, define the fibrewise multiplication of models $C^{k-1}_{r+1}$ on $\B \otimes \land W$ as 
$C^{k-1}_{r+1}(w_i) = \overline{C}^{k-1}_r(w_i)$ for each $i \not= k$, and $C^{k-1}_{r+1}(w_k) = w_k + w'_k +  P_{r+1} + \cdots$, i.e., omitting the term $P'_{2m+1}$ from $\overline{C}^{k-1}_r(w_k)$.
Now define a DG homotopy 
$$H_{2m+1} \colon \B \otimes \land W \to \B \otimes \land W\otimes \land W\otimes (t, dt)$$
by setting $H_{2m+1} = \mathrm{id}$ on $\B$,  $H_{2m+1}(w_i) = C^{k-1}_{r+1}\circ \Psi(w_i) = (\Psi\otimes_\B\Psi)\circ C^{k-1}_r(w_i)$ for $i \not= k$, and 
$$H_{2m+1}(w_k) = w_k + w'_k + \sum_{l\geq r+1}P_l + \sum b_I S_I +  P'_{2m+1}  t + \alpha dt.$$
This is a DG homotopy over and under $\B$ that makes the following diagram homotopy commute over and under $\B$:
$$
\xymatrix{\B \otimes \land W \ar[d]_{\Psi} \ar[r]^-{C^{k-1}_r}&\B \otimes \land W \otimes \land
W \ar[d]^{\Psi\otimes_\B\Psi}  \\
\B \otimes \land W \ar[r]_-{C^{k-1}_{r+1}} &\B \otimes \land W
\otimes \land W }
$$
That is, we have an equivalence  
$(D, C^{k-1}_r) \equiv^\B_\B (D, C^{k-1}_{r+1})$.

In either case, then, we have an equivalent fibrewise multiplication of models $C^k_{r+1}$ that agrees with $C^k_r$ on generators $w_i$ with $i < k$, and satisfies
$$C^k_{r+1}(w_k) = w_k + w'_k +   P_{r+1} + \cdots.$$
By induction, we have an equivalence  
$(D, C^{k-1}) \equiv^\B_\B (D, C^{k})$, where  $C^{k}$ agrees with
$C_0$ on generators through $w_k$.

For finite-dimensional $W$, the theorem follows immediately by induction over the (finitely many) generators.  In case $W$ is infinite-dimensional, note that in the above steps, in which we built an equivalence between $C^{k-1}$ and $C^{k}$, the DG homotopies and the change of generators isomorphism were all stationary, or the identity, on generators $w_i$ with $i < k$.  Since our fibrewise multiplications, and the equivalences between them,  ``stabilize" in this way, we obtain the same conclusion even when $W$ is infinite-dimensional. 
\end{proof}

If either  condition (1) or condition (2) of \thmref{LS} do not hold, then the conclusion may fail, as the following examples illustrate.

\begin{example}
When $B = *$ and $G = \Omega Y$ is  a loop space,  the conclusion of \thmref{LS} holds  if and only if the rational homotopy groups $\pi_*(G) \otimes \Q$ equipped  with the Samelson product is an abelian graded Lie algebra.  Thus taking  $G = \Omega S^2$ we obtain an example of a   fibrewise associative H-space satisfying condition (1) but not condition (2) for which the conclusion of \thmref{LS} fails. Another example is given in  \exref{ex: non-Hopf}. The free loop space $\Lambda S^2$ satisfies condition (1) but not condition (2).  The fibrewise multiplication of models $C$ described in \exref{ex: non-Hopf} is easily checked to be associative, but non-equivalent to the standard multiplication $C_0$.

On the other hand, consider the fibrewise multiplication of models $(\land (b_3) \otimes \land (w_3, w_9), D=0, C)$, where subscripts denote the degrees of generators,  and $C(w_3) = C_0(w_3)$ but $C(w_9) = w_9 + w'_9 + b_3 w_3 w'_3$.   This example satisfies condition (2) but not condition (1) of \thmref{LS}, and is likewise easily checked to be associative, but non-equivalent to the standard multiplication $C_0$. 
\end{example}

We now show how to deduce equation (\ref{eq:basic
form}) from the identity (\ref{eq:basic identity}).  We recall that
$\alpha, \beta, \gamma, \delta$ denote algebra maps
$\mathcal{B}\otimes \land W \otimes \land W\to \mathcal{B}\otimes
\land W \otimes \land W\otimes \land W$.  In fact $\alpha$ is simply
the inclusion
$$\mathcal{B}\otimes \land W \otimes \land W \to \mathcal{B}\otimes
\land W \otimes \land W\otimes 1 \subseteq \mathcal{B}\otimes \land
W \otimes \land W\otimes \land W$$
and $\delta$ the inclusion
$$\mathcal{B}\otimes \land W \otimes \land W \to \mathcal{B}\otimes
1 \otimes \land W\otimes \land W \subseteq \mathcal{B}\otimes \land
W \otimes \land W\otimes \land W.$$
The maps $\beta$ and $\gamma$ are each defined as the identity on
$\mathcal{B}$ and extended on generators $w$ and $w'$ of $\land W
\otimes \land W$ by: $\beta(w) = w + w'$ and $\beta(w') = w''$;
$\gamma(w) = w$ and $\gamma(w') = w' + w''$.

\begin{proposition}\label{prop:alpha beta gamma delta}
Suppose $\chi \in \mathcal{B}\otimes\land W \otimes\land W $ is of
homogeneous length $r \geq 3$ and ``mixed" in $W\oplus W$, i.e.,
$$\chi \in \sum_{1 \leq s \leq r-1}
\mathcal{B}\otimes\land^{s} W \otimes\land^{r-s} W.$$
If $\chi$ satisfies
\begin{equation}\label{eq:alpha beta gamma delta}
\alpha(\chi) + \beta(\chi) = \gamma(\chi) + \delta(\chi),
\end{equation}
then we may write $\chi$ in the form
$$\chi = \sum b_I (S_I - w_I - w'_I)$$
for suitable $b_I \in \mathcal{B}$, where the sum is over sequences
$I = (i_1, i_2, \dots , i_r)$ of length $r$ that are strictly
increasing, i.e., for which $i_1 < i_2 < \cdots < i_r$.
\end{proposition}

We prove this via a sequence of subsidiary lemmas.  To begin, notice
that $\alpha$, $\beta$, $\gamma$ and $\delta$ do not alter
subscripts of generators of $\land W \otimes \land W$.  Using anti-commutativity, we  may write $\chi$ 
  as a sum of   $\chi_I = \sum_{\varepsilon} a^{\varepsilon}_I w^{\varepsilon}_I$ as defined above the statement of Theorem \ref{LS}. Further,  because the $\chi_I$
in such a sum either satisfy or fail to satisfy the  identity (\ref{eq:alpha beta gamma
delta}) independently of each other,  it is sufficient to show the
result for each $\chi_I$ separately.

First, we eliminate the possibility of any $\chi_I$ that has a
repeated subscript satisfying the identity of the hypothesis. Call
the sequence $I$ the \emph{subscript sequence} of $\chi_I$, and
recall that it is a non-decreasing sequence $I = (i_1, \dots, i_r)$
with $i_1 \leq \cdots \leq i_r$.  Since we are assuming that each
generator of $W$ is of odd degree, and since each $w_i^{\epsilon}$
is either in $1\otimes\land W \otimes1$ or
$1\otimes1\otimes\land W $, any $\chi_I$ may have any one subscript
repeated at most once, that is, at most two consecutive subscripts
$i_s$ and $i_{s+1}$ may agree with each other and, if they do, then
we must have $i_{s-1} < i_s = i_{s+1} < i_{s+2}$.

\begin{lemma}\label{lem:no repeats}
Suppose that
$$\chi_I = \sum_{\varepsilon} a^{\varepsilon}_I w^{\varepsilon}_I,$$
has subscript sequence $I$ of length $r \geq 3$ with a repeated
entry, so that $I = (i_1, \dots, i_s, i_{s+1}, \dots, i_r)$ with
$i_s = i_{s+1}$. If $\chi_I$ satisfies
\begin{equation}\label{eq:basic with repeats}
\alpha(\chi_I) + \beta(\chi_I) = \gamma(\chi_I) + \delta(\chi_I),
\end{equation}
then $\chi_I = 0$.
\end{lemma}

\begin{proof}
We may use the anti-commutativity of $\mathcal{B}\otimes \land W
\otimes \land W$ to write this $\chi_I$ as $\chi_I =
\chi_{\widehat{I}} w_{i_s} w'_{i_s}$,  where
$$\chi_{\widehat{I}} = \sum_{\varepsilon} \widehat{a}_I^{\varepsilon} w^{\varepsilon}_{\widehat{I}},$$
with ${\widehat{I}} = (i_1, \dots, \widehat{i_s}, \widehat{i_s},
\dots, i_r)$ is the subscript sequence of length $r-2$ with the
repeated entries removed.  Notice that we allow for further repeats
in the remaining sequence ${\widehat{I}}$.  Expanding identity
(\ref{eq:basic with repeats}), we obtain
$$\alpha\big(\chi_{\widehat{I}}\big) w_{i_s} w'_{i_s} + \beta\big(\chi_{\widehat{I}}\big) (w_{i_s}
+ w'_{i_s}) w''_{i_s}  = \gamma\big(\chi_{\widehat{I}}\big) w_{i_s}
(w'_{i_s} + w''_{i_s})  + \delta\big(\chi_{\widehat{I}}\big)
w'_{i_s} w''_{i_s}.$$
Bearing in mind that the subscript $i_s$ does not occur in
$\widehat{I}$, we may equate separately the terms from this equation
in $w_{i_s} w'_{i_s}$, $w_{i_s} w''_{i_s}$, and $w'_{i_s}
w''_{i_s}$.  Doing so gives the three equations
$$\alpha\big(\chi_{\widehat{I}}\big) = \gamma\big(\chi_{\widehat{I}}\big), \quad
\beta\big(\chi_{\widehat{I}}\big) =
\gamma\big(\chi_{\widehat{I}}\big), \quad \mathrm{and} \quad
\beta\big(\chi_{\widehat{I}}\big) =
\delta\big(\chi_{\widehat{I}}\big),$$
respectively.  Combining these three, we have that
$\alpha\big(\chi_{\widehat{I}}\big) =
\delta\big(\chi_{\widehat{I}}\big)$, from which it follows that
$\chi_{\widehat{I}} =0$. For $\alpha\big(\chi_{\widehat{I}}\big) \in
\mathcal{B}\otimes \land W \otimes \land W \otimes 1$, whereas
$\delta\big(\chi_{\widehat{I}}\big) \in \mathcal{B}\otimes 1 \otimes
\land W \otimes \land W$.  If they agree, we must have
$\alpha\big(\chi_{\widehat{I}}\big)$ in their intersection, namely
$\mathcal{B}\otimes 1 \otimes \land W \otimes 1$, and hence
$\chi_{\widehat{I}} \in \mathcal{B}\otimes 1 \otimes \land W$. This
in turn implies that $\delta\big(\chi_{\widehat{I}}\big) \in
\mathcal{B}\otimes 1 \otimes 1 \otimes \land W$.  Once again, since
$\alpha\big(\chi_{\widehat{I}}\big) =
\delta\big(\chi_{\widehat{I}}\big)$, then we must have
$\alpha\big(\chi_{\widehat{I}}\big)$ in the intersection, namely
$\mathcal{B}\otimes 1 \otimes 1 \otimes 1$. So, in fact, we have
$\chi_{\widehat{I}} \in \mathcal{B}\otimes 1 \otimes 1$. But
$\chi_{\widehat{I}}$ has indexing sequence of length at least $1$,
and it follows that $\chi_{\widehat{I}} = 0$.
\end{proof}

\begin{remark}
Notice that the restriction $r\geq 3$ is necessary in
\propref{prop:alpha beta gamma delta} and \lemref{lem:no repeats}.
Indeed, any term of the form $\chi = a w_i w'_j$, in which $r=2$ and
$i \leq j$ or $i > j$, satisfies the identity $\alpha(\chi) +
\beta(\chi) = \gamma(\chi) + \delta(\chi)$, as may easily be
checked.
\end{remark}

When combined with the remainder of the argument, the preceding
result justifies the last clause of the statement of
\propref{prop:alpha beta gamma delta}.  The remaining arguments take
up the main point, with $I = (i_1, i_2, \dots , i_r)$ and $i_1 < i_2
< \cdots < i_r$.

\begin{lemma}\label{lem:beta=gamma}
Fix $I = (i_1, i_2, \dots , i_r)$ with $i_1 < i_2 < \cdots < i_r$
and $r \geq 1$.
\begin{itemize}
\item[(A)] If
\begin{equation}\label{eq:beta=gamma}
\beta(\chi_I) = \gamma(\chi_I),
\end{equation}
then $\chi_I =  a S_I$, where $a \in \mathcal{B}$ and $S_I =
(w_{i_1} + w'_{i_1}) \cdots (w_{i_r} + w'_{i_r})$.

\item[(B)] If $\beta(\chi_I) = 0$, then
$\chi_I = 0$.

\item[(C)] If $\beta(\chi_I) = \delta(\chi_I)$, then
$\chi_I \in \mathcal{B}\otimes 1 \otimes \land W$.
\end{itemize}
\end{lemma}

\begin{proof}
(A) We argue by induction on $r$.  Induction starts with $r=1$,
where we have $\chi_I = a w_{i_1} + a' w'_{i_1}$, with $a, a' \in
\mathcal{B}$. We compute that $\beta(\chi_I) = a w_{i_1} + a
w'_{i_1} + a'w''_{i_1}$, and $\gamma(\chi_I) = a w_{i_1} + a'
w'_{i_1} + a'w''_{i_1}$.  Then (\ref{eq:beta=gamma}) implies that $a
= a'$, and we have $\chi_I = a S_{i_1}$.

For the induction step, with $r\geq 2$, write $\chi_I = \chi_1
w_{i_r} + \chi_2 w'_{i_r}$, with $\chi_1$ and $\chi_2$ of length
$r-1$, and expand (\ref{eq:beta=gamma}) as
$$\beta(\chi_1) w_{i_r} + \beta(\chi_1) w'_{i_r} + \beta(\chi_2)
w''_{i_r} = \gamma(\chi_1) w_{i_r} + \gamma(\chi_2) w'_{i_r} +
\gamma(\chi_2) w''_{i_r}.$$
Since the subscript $i_r$ is strictly the highest that occurs in
$I$, we may equate separately the terms from this equation in
$w_{i_r}$, $w'_{i_r}$, and $w''_{i_r}$. Doing so gives the three
equations
$$\beta(\chi_1) = \gamma(\chi_1), \quad
\beta(\chi_1) = \gamma(\chi_2), \quad \mathrm{and} \quad
\beta(\chi_2) = \gamma(\chi_2),$$
respectively. The first and third of these, with the induction
hypothesis, yield $\chi_1 = a S_{\widehat{I}}$ and $\chi_2 = a'
S_{\widehat{I}}$, where $\widehat{I} = (i_1, \dots, i_{r-1})$. Since
$\beta(S_I) = \gamma(S_I)$, the middle of the three equations
consequently implies $a=a'$.  Hence $\chi_I = a S_{\widehat{I}}
w_{i_r} + a S_{\widehat{I}} w'_{i_r} = a S_{\widehat{I}}(w_{i_r} +
w'_{i_r}) = a S_{I}$.  Induction is complete and the result follows.

Parts (B) and (C) are proved by entirely analogous arguments.  Part
(B) is used in establishing the induction step of the proof of part
(C).
\end{proof}

\begin{lemma}\label{lem:alpha beta gamma}
Fix $I = (i_1, i_2, \dots , i_r)$ with $i_1 < i_2 < \cdots < i_r$
and $r \geq 1$.
\begin{itemize}
\item[(A)] If
\begin{equation}\label{eq:alpha + beta=gamma}
\alpha(\chi_I) + \beta(\chi_I) = \gamma(\chi_I),
\end{equation}
then $\chi_I =  a (S_I - w_I)$, where $a \in \mathcal{B}$, $S_I =
(w_{i_1} + w'_{i_1}) \cdots (w_{i_r} + w'_{i_r})$, and $w_I =
w_{i_1} \cdots w_{i_r}$.

\item[(B)] If $\beta(\chi_I) = \gamma(\chi_I) + \delta(\chi_I)$,
then $\chi_I =  a (S_I - w'_I)$.
\end{itemize}
\end{lemma}

\begin{proof}
(A) We argue by induction on $r$.  Induction starts with $r=1$,
where we have $\chi_I = a w_{i_1} + a' w'_{i_1}$.  We compute that
$\alpha(\chi_I) + \beta(\chi_I) = a w_{i_1} + a' w'_{i_1} + a
w_{i_1} + a w'_{i_1} + a'w''_{i_1}$, and $\gamma(\chi_I) = a w_{i_1}
+ a' w'_{i_1} + a'w''_{i_1}$.  Then (\ref{eq:alpha + beta=gamma})
implies that $a = 0$, and we have $\chi_I = a'w'_{i_1} =  a'(
S_{i_1} - w_{i_1})$ as desired.

For the induction step, with $r\geq 2$, write $\chi_I = \chi_1
w_{i_r} + \chi_2 w'_{i_r}$, with $\chi_1$ and $\chi_2$ of length
$r-1$, and expand (\ref{eq:alpha + beta=gamma}) as
$$\alpha(\chi_1)
w_{i_r} + \alpha(\chi_2) w'_{i_r} + \beta(\chi_1) w_{i_r} +
\beta(\chi_1) w'_{i_r} + \beta(\chi_2) w''_{i_r} = \gamma(\chi_1)
w_{i_r} + \gamma(\chi_2) w'_{i_r} + \gamma(\chi_2) w''_{i_r}.$$
Since the subscript $i_r$ is strictly the highest that occurs in
$I$, we may equate separately the terms from this equation in
$w_{i_r}$, $w'_{i_r}$, and $w''_{i_r}$. Doing so gives the three
equations
$$\alpha(\chi_1) + \beta(\chi_1) = \gamma(\chi_1), \quad
\alpha(\chi_2) + \beta(\chi_1) = \gamma(\chi_2), \quad \mathrm{and}
\quad \beta(\chi_2) = \gamma(\chi_2),$$
respectively. The third of these, with \lemref{lem:beta=gamma},
yields $\chi_2 = a S_{\widehat{I}}$, where $\widehat{I} = (i_1,
\dots, i_{r-1})$. The first, with the induction hypothesis, yields
$\chi_1 = b (S_{\widehat{I}} - w_{\widehat{I}})$. Consequently, the
middle of these three equations yields
$$a \alpha(S_{\widehat{I}}) + b \beta(S_{\widehat{I}}) - b \beta(w_{\widehat{I}})
= a \gamma(S_{\widehat{I}}).$$
Write $T_i$ for the trinomial $w_i + w'_i + w''_i \in \mathcal(B)
\otimes \land W \otimes \land W \otimes \land W$, and extend this
notation as we have for the binomials to write $T_I = T_{i_1} \cdots
T_{i_r}$.  Observe that we have $\beta(w_{\widehat{I}}) =
S_{\widehat{I}}$, whereas $\beta(S_{\widehat{I}}) =
\gamma(S_{\widehat{I}}) = T_{\widehat{I}}$.  The previous equation,
therefore, may be written as
$$(a - b) \alpha(S_{\widehat{I}}) = (a-b)T_{\widehat{I}}.$$
By equating coefficients in $\delta(w_{\widehat{I}})$, for instance,
we see that $a = b$.   Therefore, we have $\chi_I = b
(S_{\widehat{I}} - w_{\widehat{I}}) w_{i_r} + b S_{\widehat{I}}
w'_{i_r} = b (S_I - w_I)$ as desired. Induction is complete and the
result follows.

A similar argument establishes (B).
\end{proof}

\begin{proof}[Proof of \propref{prop:alpha beta gamma delta}]
By \lemref{lem:no repeats}, we may assume that the subscript
sequence $I = (i_1, \dots, i_r)$ is strictly increasing. Suppose
$r\geq 3$, write $\chi_I = \chi_1 w_{i_r} + \chi_2 w'_{i_r}$, with
$\chi_1$ and $\chi_2$ of length $r-1 \geq 2$, and expand
(\ref{eq:alpha beta gamma delta}) as
$$\begin{aligned}
\alpha(\chi_1) w_{i_r} + \alpha(\chi_2) w'_{i_r} + &\beta(\chi_1)
w_{i_r} + \beta(\chi_1) w'_{i_r} + \beta(\chi_2)
w''_{i_r}\\
& = \gamma(\chi_1) w_{i_r} + \gamma(\chi_2) w'_{i_r} +
\gamma(\chi_2) w''_{i_r} + \delta(\chi_1) w'_{i_r} + \delta(\chi_2)
w''_{i_r}.
\end{aligned}
$$
Since the subscript $i_r$ is strictly the highest that occurs in
$I$, we may equate separately the terms from this equation in
$w_{i_r}$, $w'_{i_r}$, and $w''_{i_r}$. Doing so gives the three
equations
$$\begin{aligned}
\alpha(\chi_1) + \beta(\chi_1) &= \gamma(\chi_1), \\
\alpha(\chi_2) + \beta(\chi_1) &= \gamma(\chi_2) + \delta(\chi_1),\\
\beta(\chi_2) &= \gamma(\chi_2) + \delta(\chi_2),
\end{aligned}
$$
respectively. The first and third of these, with \lemref{lem:alpha
beta gamma}, yield $\chi_1 = a (S_{\widehat{I}} - w_{\widehat{I}})$
and $\chi_2 = a' (S_{\widehat{I}} - w'_{\widehat{I}})$, where
$\widehat{I} = (i_1, \dots, i_{r-1})$ and $a, a' \in \mathcal{B}$.
Consequently, the middle of these three equations yields
$$
\begin{aligned}
a' \alpha(S_{\widehat{I}}) - a' \alpha(w'_{\widehat{I}}) &+ a \beta(S_{\widehat{I}}) - a \beta(w_{\widehat{I}}) \\
 &= a' \gamma(S_{\widehat{I}}) - a' \gamma(w'_{\widehat{I}}) + a \delta(S_{\widehat{I}}) - a \delta(w_{\widehat{I}}).
\end{aligned}
$$
We may rewrite this, substituting $\alpha(S_{\widehat{I}}) =
\beta(w_{\widehat{I}})$, $\alpha(w'_{\widehat{I}}) =
\delta(w_{\widehat{I}})$, $\gamma(S_{\widehat{I}}) =
\beta(S_{\widehat{I}})$, and $\gamma(w'_{\widehat{I}}) =
\delta(S_{\widehat{I}})$, as
$$
\begin{aligned}
a' \beta(w_{\widehat{I}}) - a' \delta(w_{\widehat{I}}) &+ a \beta(S_{\widehat{I}}) - a \beta(w_{\widehat{I}}) \\
 &= a' \beta(S_{\widehat{I}}) - a' \delta(S_{\widehat{I}}) + a \delta(S_{\widehat{I}}) - a \delta(w_{\widehat{I}}).
\end{aligned}
$$
Now this may be re-arranged to the equation
$$
\beta\big( (a-a')( S_{\widehat{I}} - w_{\widehat{I}})\big) =
\delta\big( (a-a')( S_{\widehat{I}} - w_{\widehat{I}})\big).
$$
Part (C) of \lemref{lem:beta=gamma} now implies that $(a-a')(
S_{\widehat{I}} - w_{\widehat{I}}) \in \mathcal{B}\otimes 1 \otimes
\land W$. But since $\widehat{I}$ is of length $\geq 2$, we cannot
have $S_{\widehat{I}} - w_{\widehat{I}} \in \mathcal{B}\otimes 1
\otimes \land W$.  Hence we must have $a = a' = b$, say, and thus
$\chi_I = b (S_{\widehat{I}} - w_{\widehat{I}}) w_{i_r} + b
(S_{\widehat{I}} - w'_{\widehat{I}}) w'_{i_r} = b (S_{\widehat{I}} -
w_{\widehat{I}} - w'_{\widehat{I}})$ as desired.
\end{proof}

We can now deduce our  second main result. 

\begin{proof}[Proof of Theorem \ref{main1}]
Let $X$ be a fibrewise associative H-space satisfying the conditions in Theorem \ref{main1}.   Then the relative   model $\B \to \B \otimes \land W$ for $X \to B$  satisfies the conditions of Theorem \ref{LS}. By that result,  the fibrewise multiplication of models $C$ induced by the fibrewise multiplication  $m$ is equivalent to $C_0$. By Theorem \ref{eqtrivial}, $X$ is rationally fibrewise H-trivial.  \end{proof}

\section{Rational H-triviality of some Fibrewise groups} \label{sec:applications}
 We show how Theorem \ref{main1} may be applied  to   a  standard construction of fibrewise groups (see \cite[p.14]{CJ}).   
 Let $H$ be a  homotopy finite  topological group and  $p \colon E \to B$   a principal $H$-bundle.  We assume   $p$ is a pull-back of the universal principal $H$-bundle $EH \to BH$ by a map $h \colon B \to BH.$     
 Suppose $H$ acts    on  another homotopy finite     group $G$ via   a continuous homomorphism $\alpha \colon H \to \mathrm{Aut}(G).$ Then $H$ acts on   $E \times G$ by the  diagonal action and we obtain a fibrewise group   $$ q \colon E \times_H G \to B$$ with fibre $G$. 
 Here   $E \times_H G$ is  the quotient of $E \times G$ by the  action and $q$ is induced from $p$.  
 The canonical example    is the adjoint bundle $\Ad(p) \colon E \times_G G \to B$   with $\alpha = \Ad_G \colon G \to \mathrm{Aut}(G)$ the adjoint action.   
 
 Given an action $\alpha \colon G \to \mathrm{Aut}(H)$,  we may perform the same construction  on the universal $H$-bundle to obtain  a fibrewise group  $EH \times_H G$  over  $BH$. Notice that the fibrewise group $E \times_H G$ over $B$ may be taken to be the pull-back by $h$ of the fibrewise group $EH \times_H G$ over $BH.$  

\begin{theorem} \label{thm:app}  Let $H$ and $G$ be homotopy finite topological groups, $p \colon E \to B$ a principal $H$-bundle and $H \to \mathrm{Aut}(G)$ a continuous homomorphism.  Suppose $p$ is a pull-back of the universal $H$-bundle $EH \to BH$ and that    $EH \times_H G$ is a nilpotent space.  Then the fibrewise group $E \times_G H$ over $B$ is rationally fibrewise H-trivial.  
\end{theorem}
\begin{proof} 
Since   $H$ is of the homotopy type of finite CW complex,     $H^*(BH; \Q)$ is a polynomial algebra on even generators.  Since the other conditions of Theorem \ref{main1} hold by hypothesis,  we obtain that the fibrewise group $ EH \times_H G$ over $BH$ is    rationally fibrewise H-trivial.  Now observe that the uniqueness of fibrewise localization    implies    fibrewise H-triviality is   preserved by pull-backs.  The result follows.
\end{proof} 
 When $G$ and $H$ are simply connected,   the nilpotence of the space $EH \times_H G$ in  Theorem \ref{thm:app} is automatic.    We observe that  Theorem \ref{thm:app} applies without this restriction in the special case of the adjoint bundle.  The following result contains \cite[Prop.2.2]{CS}.  
 
 \begin{theorem}  \label{thm:gauge} Let $G$ be a homotopy finite topological group and  $ p \colon E \to B$ a principal $G$-bundle.  Suppose $p$ is a pull-back of the universal $G$-bundle $EG \to BG$.  Then the adjoint bundle $\Ad(p)$  over $B$ is  rationally fibrewise H-trivial. 
 \end{theorem}
 \begin{proof}
It suffices to prove $EG \times_G G$ is a nilpotent space where $G$ acts on itself via the adjoint action.  We  here have the identification $ EG \times_G G \simeq \map(S^1, BG)$  (cf. \cite[Lem.A.1]{KSS}). Since $S^1$ is a finite CW complex,  nilpotence is given  by  \cite[Th.II.3.11]{HMR}.     
 \end{proof}

We give one further example  based on work in   \cite{KSS}. 
 
 \begin{example} \label{ex:univ}
Let  $G$ be a compact Lie group.  Let  $P(G) = G/\mathcal{Z}(G)$ be the projectification.  Suppose given  a principal $P(G)$-bundle $p \colon E \to B$  which is the pull-back of  the universal $P(G)$-bundle.  Taking $P(G) \to \mathrm{Aut}(G)$ to be the induced adjoint action, we obtain 
 $$\Pad(p) = E \times_{P(G)} G \to B,$$
 the {\em projective adjoint bundle}.   The universal example is of the form $EP(G) \times_G G\to BP(G)$. By  \cite[Lem.5.11]{KSS}, $EP(G) \times_G G$ is a nilpotent space.    Theorem \ref{thm:app} thus applies to show   $$\hbox{$\Pad(p)$ is rationally fibrewise H-trivial.}$$      \end{example}

Finally,  given a fibrewise H-space $X$ over $B,$ let  $\Gamma(X)$ denote the space of sections of $X$  and $\Gamma(X)_\circ$ the path-component of the basepoint $\sigma$ in $\Gamma(X)$.  Then $\Gamma(X)_\circ$ is a connected H-space with point-wise multiplication of sections.     
Taking $X = E \times_G G$  to be the adjoint bundle over $B$ corresponding to a principal $G$-bundle $p \colon E \to B$, we have the identity $\G(p) \cong \Gamma(X)$  where $\G(p)$  is the {\em gauge group}  of the bundle \cite[Eq.2.2]{AB}.  
 
 \begin{theorem} \label{thm:sec} Let    $X$ be a fibrewise H-space with $B$ and  $G$   homotopy finite. Suppose $X$ is rationally fibrewise H-trivial.  Then   $\Gamma(X)_\circ$ is rationally H-commutative    and there is a rational H-equivalence
 $$ \Gamma(X)_\circ \simeq_\Q \map(B, G)_\circ. $$
 \end{theorem}
 \begin{proof}   
   By \cite[Th.5.3]{Mol}, fibrewise rationalization induces a rationalization map $$\Gamma(X)_\circ \to \Gamma(X_{(\Q)})_\circ.$$  Since $X_{(\Q)}$ is fibrewise H-trivial, the rational H-equivalence above follows.  Then  since $G$ is homotopy finite, $G_\Q$ is H-commutative and so $\map(B, G_\Q)_\circ$ is also.  
 \end{proof}
  
 Let $B$ be a finite complex and $G$ a homotopy finite topological group. Applying Theorems \ref{thm:gauge} and \ref{thm:sec},     we retrieve the fact that the gauge groups $\G(p)$ arising from  principal $G$-bundles $p \colon E \to B$   are all  rationally H-homotopy commutative and pairwise rationally H-homotopy equivalent \cite[Th.5.6]{KSS}.

  \bibliographystyle{amsplain}

\end{document}